\documentclass{gtmon_a}
\pdfoutput=1

\usepackage{pinlabel}

\proceedingstitle{The Zieschang Gedenkschrift}
\conferencestart{5 September 2007}
\conferenceend{8 September 2007}
\conferencename{Conference in honour of Heiner Zieschang}
\conferencelocation{Toulouse, France}

\editor{Michel Boileau}
\givenname{Michel}
\surname{Boileau}

\editor{Martin Scharlemann}
\givenname{Martin}
\surname{Scharlemann}

\editor{Richard Weidmann}
\givenname{Richard}
\surname{Weidmann}

\title[Non compact Euclidean cone 3--manifolds with cone angles less
than $2\pi$]{Non compact Euclidean cone 3--manifolds\\with cone angles less
than $2\pi$}

\author{Daryl Cooper}
\givenname{Daryl}
\surname{Cooper}
\address{Department of Mathematics\\
University of California at Santa Barbara\\\newline
Santa Barbara CA 93106\\
USA}
\email{cooper@math.ucsb.edu}
\urladdr{}

\author{Joan Porti}
\givenname{Joan}
\surname{Porti}
\address{Departament de Matem\`atiques\\
Universitat Aut\`onoma de Barcelona\\\newline
E-08193 Bellaterra\\
Spain}
\email{porti@mat.uab.cat}
\urladdr{}

\volumenumber{14}
\issuenumber{}
\publicationyear{2008}
\papernumber{10}
\startpage{173}
\endpage{192}

\doi{}
\MR{}
\Zbl{}

\arxivreference{}


\received{31 May 2006}
\revised{19 April 2007}
\accepted{23 April 2007}
\published{29 April 2008}
\publishedonline{29 April 2008}
\proposed{}
\seconded{}
\corresponding{}
\version{}

\makeatletter
\def\cnewtheorem#1[#2]#3{\newtheorem{#1}{#3}[section]
\expandafter\let\csname c@#1\endcsname\c@dfn}
\makeatother

\AtBeginDocument{\let\bar\wbar\let\tilde\wtilde\let\hat\what}

\theoremstyle{definition}
\newtheorem{dfn}{Definition}[section]
\cnewtheorem{add}[dfn]{Addendum}
\cnewtheorem{ass}[dfn]{Assumption}
\cnewtheorem{conc}[dfn]{Conclusion}
\cnewtheorem{cons}[dfn]{Consequence}
\cnewtheorem{crit}[dfn]{Criterion}
\cnewtheorem{exm}[dfn]{Example}
\cnewtheorem{gen}[dfn]{Generalisation}
\cnewtheorem{ques}[dfn]{Question}
\cnewtheorem{rem}[dfn]{Remark}
\cnewtheorem{sit}[dfn]{Situation}
\cnewtheorem{spec}[dfn]{Special case}

\theoremstyle{plain}
\cnewtheorem{claim}[dfn]{Claim}
\cnewtheorem{conj}[dfn]{Conjecture}
\cnewtheorem{cor}[dfn]{Corollary}
\cnewtheorem{lem}[dfn]{Lemma}
\cnewtheorem{prop}[dfn]{Proposition}
\cnewtheorem{sublem}[dfn]{Sublemma}
\cnewtheorem{thm}[dfn]{Theorem}

\makeautorefname{exm}{Example}

\newcommand{\SO}{\mathit{SO}}
\newcommand{\Spin}{\mathit{Spin}}

\begin{document}

\begin{htmlabstract}
We describe some properties of noncompact Euclidean cone manifolds with
cone angles less than c<2&pi; and singular locus a submanifold. More
precisely, we describe its structure outside a compact set.  As a
corollary we classify those with cone angles <3&pi;/2
and those with all cone angles =3&pi;/2.
\end{htmlabstract}

\begin{abstract}
We describe some properties of noncompact Euclidean cone manifolds with
cone angles less than $c<2\pi$ and singular locus a submanifold. More
precisely, we describe its structure outside a compact set.  As a
corollary we classify those with cone angles $<{3\pi}/2$
and those with all cone angles $=3\pi/2$.
\end{abstract}

\begin{asciiabstract}
We describe some properties of noncompact Euclidean cone manifolds with
cone angles less than c<2pi and singular locus a submanifold. More
precisely, we describe its structure outside a compact set.  As a
corollary we classify those with cone angles <3pi/2
and those with all cone angles =3\pi/2.
\end{asciiabstract}

\maketitle

\section{Introduction}

In this paper we study non-compact orientable Euclidean cone 3--manifolds
with cone angles  less than $2 \pi$.
 When the cone angles are $\leq \pi$ these manifolds are classified:
they play a key role in the proof of the orbifold theorem, as  they are
rescaled limits of collapsing sequences of hyperbolic or spherical cone
manifolds (see Boileau--Porti~\cite{BP} and Cooper--Hodgson--Kerckhoff
\cite{CHK}).  The aim of this paper is to  have  some understanding when
the cone angles lie between $\pi$ and $2\pi$.

We will fix an \emph{upper bound of the cone angles} $c<2\pi$.
The reason is that if we only impose cone angles  $<2\pi$, the
singular locus can have infinitely many components.

For simplicity, we will also restrict to the case where the \emph{singular
set is a submanifold}.

{Besides the isometry type of a cone manifold $E$, we are
also interested in the topology of the pair $(\vert E\vert , \Sigma)$,
where $\vert E\vert$ denotes its underlying topological space and $\Sigma$
its singular locus.}

The first tool to study those cone manifolds is the soul theorem of
Cheeger and Gromoll~\cite{CG}, or its cone manifold version.  The soul
can have dimension 0, 1 or 2. If the dimension is 1 or 2, then the cone
manifold is easy to describe, the difficulties arise when the soul is
zero dimensional (that is, just a point).  {The reader familiar
with~\cite{BP} should be aware that the definiton of the soul used in
this paper differs form that used there, where it was adapted to cone
manifolds with cone angle $\leq \pi$.}

The following proposition says that the singular locus is
unknotted, provided there are no compact singular components.

\begin{prop} \label{prop:sigmaunkotted}
Let $E$ be a Euclidean cone 3--manifold with cone angles $\leq c<
2\pi $ and soul a point. Assume that its singular locus $\Sigma$ is a
non-empty submanifold.  If all components of $\Sigma$ are non-compact,
then the pair {$(\vert E \vert ,\Sigma)$} is homeomorphic
to $\mathbb{R}^3$ with some straight lines.
\end{prop}

When there are compact singular components, we have a nice description
away from a compact set. We start with some examples.  The \emph{angle
defect} of a singularity is $2\pi$ minus the cone angle.

\begin{exm}\label{ex1}
Consider   a Euclidean cone metric on $D^2$ with totally geodesic
boundary. Such a metric exists if and only if the sum of the cone
angle defects is $2\pi$. This metric can be enlarged to a complete
metric on the plane by adding a flat cylinder $S^1\times [0,\infty)$,
where $S^1\times\{0\}=\partial D^2$.  Take the product with $\mathbb{R}$, so that we get a three dimensional manifold with closed parallel
geodesics. Some of those geodesics can be easily replaced by singular
geodesics, provided that the cone angle defects add up to $\pi$. See
\fullref{fig:simplest}.
\end{exm}

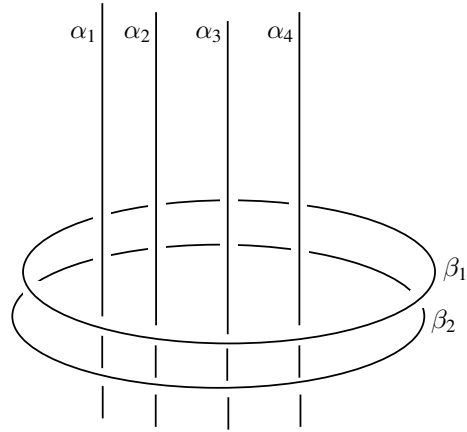
\begin{figure}
\begin{center}
\labellist\small\hair1pt
\pinlabel {$\alpha_1$} [r] at 45 200
\pinlabel {$\alpha_2$} [r] at 72 200
\pinlabel {$\alpha_3$} [r] at 108 200
\pinlabel {$\alpha_4$} [r] at 144 200
\pinlabel {$\beta_1$} [l] at 216 82
\pinlabel {$\beta_2$} [l] at 210 56
\endlabellist
\includegraphics[scale=.75]{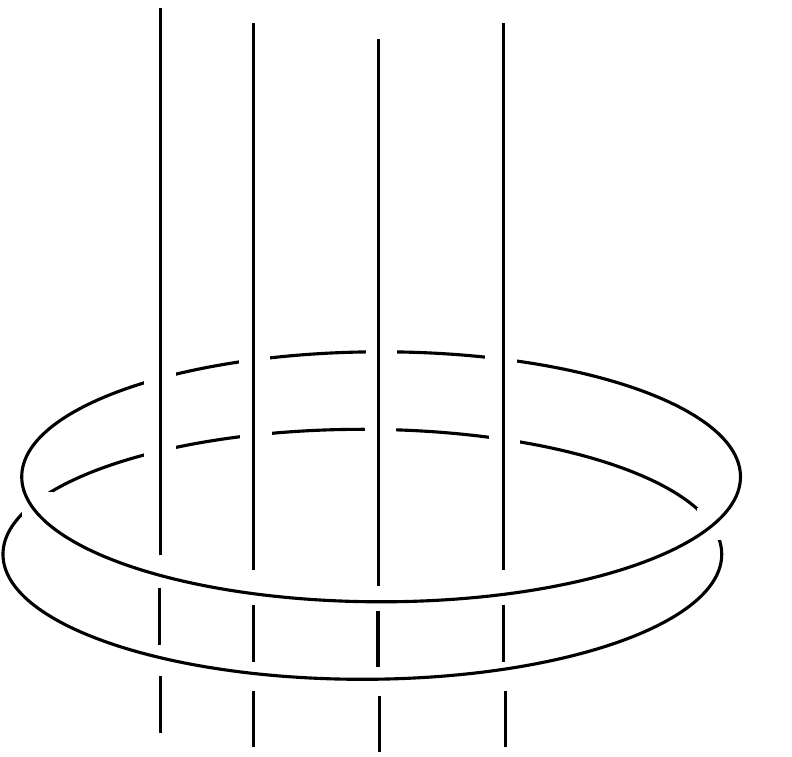}
\end{center}
 
\caption{The singular locus of a cone manifold as in \fullref{ex1}. The underlying space is $\mathbb{R}^3$, and 
the  angle defects satisfy
   $\sum(2\pi-\alpha_i)=2\pi$  and $\sum(2\pi-\beta_i)\leq \pi$.}
\label{fig:simplest}
\end{figure}

The topology of the pair {$(\vert E\vert,\Sigma)$} is more involved in the next example,
but it can still be described in terms of rational tangles.

\begin{exm}~\label{ex2}
The product $[0,1]\times \mathbb{R}^2$ is bounded by two parallel
planes. Take a geodesic on each plane such that,
when parallel transported, they intersect with an angle
equal to a rational multiple of $2\pi$.
Consider the cone manifold obtained by folding these planes along
these lines, so that the folding  lines become the singular locus, with
cone angle $\pi$. Then, the foliation by segments
$[0,1]\times\{*\}$ gives rise to a foliation by geodesic circles,
or intervals with end-points in the singular locus. Again, some of the
closed geodesics can be replaced by singular geodesics with small
cone angle defect. The group of transformations in the plane
generated by reflections on the two lines is a dihedral group with
$2n$ elements, and the sum of the cone angle defects now is
bounded above by $\pi/n$. See \fullref{fig:rational}.

This example can be further perturbed to replace the edges with
cone angle $\pi$ by several singular edges with cone angle defects
whose sum is $\pi$.
\end{exm}

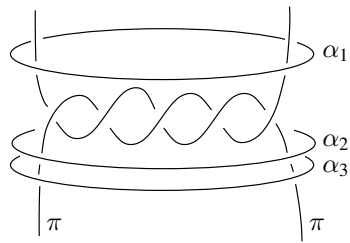
\begin{figure}[ht!]
\begin{center}
\labellist\small
\pinlabel {$\pi$} [l] at 22 14
\pinlabel {$\pi$} [l] at 220 14
\pinlabel {$\alpha_1$} [l] at 230 143
\pinlabel {$\alpha_2$} [l] at 230 76
\pinlabel {$\alpha_3$} [l] at 230 56
\endlabellist
\includegraphics[scale=.50]{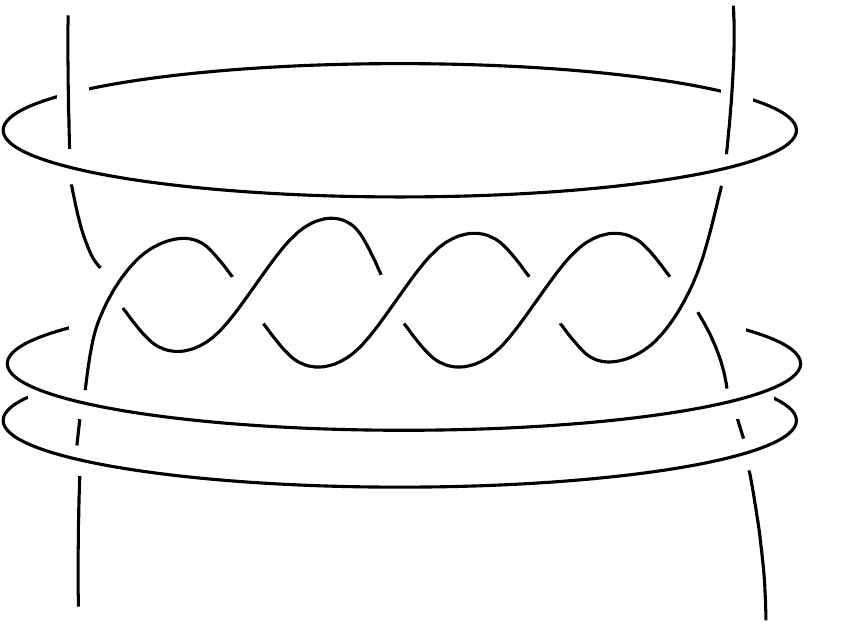}
\end{center}
    \caption{The singular locus of a cone manifold as 
    in \fullref{ex2}. $\sum(2\pi{-}\alpha_i){\leq}\frac{\pi}n$}
\label{fig:rational}
\end{figure}

Notice  that away from a compact set both examples are similar.

\begin{thm}
\label{thm:asymptotic}
 Let $E$ be a non-compact Euclidean cone 3--manifold with cone
angles $\leq c<2\pi$ and such that $\Sigma$ is a submanifold.
Assume that the soul of $E$ is a point and that it has a compact
singular component.
Then there exists a compact subset $K$ such that:
\begin{enumerate}
\item either $K=D_1=D_2$ or $\partial K=D_1\cup_{\partial} D_2$, where
$D_1$ and $D_2$ are totally geodesic discs {with singular points and}
 with geodesic boundary $\partial
D_1=\partial D_2=D_1\cap D_2$.
\item $E-\operatorname{int}(K)$ can be
decomposed isometrically in three product pieces:  $ D_1\times
[0,+\infty)$, $ D_2\times [0,+\infty)$ and  ${\mathcal E^2}\times S^1$, where
{${\mathcal E^2}$} denotes a
two dimensional Euclidean sector (that is, its boundary is two half
lines) with {singular} points. The pieces  are glued so that $D_i\times
[0,\infty)\cap K=D_i\times \{0\}$ and $\partial {\mathcal E^2}\times
S^1=\partial D_1\times [0,+\infty)\cup_{\partial D_i\times\{0\}}
\partial D_2\times [0,+\infty)$.
\item The dihedral angle between the discs $D_1$ and $D_2$  in $\partial K$ is
$\leq $ $\pi$ minus the sum of the angle defects in the sector
${\mathcal E^2}$.
\end{enumerate}
 \end{thm}

\begin{figure}[ht!]
\begin{center}
\labellist\small
\pinlabel {$D_1$} [b] at 92 90
\pinlabel {$D_2$} [t] at 92 68
\pinlabel {$K$} [b] at 180 76
\pinlabel {$C$} [r] at 68 82
\pinlabel {$\mathcal{E}^2\times S^1$} [l] at 0 104
\pinlabel {$\alpha_1$} [b] at 116 154
\pinlabel {$\alpha_2$} [b] at 136 154
\pinlabel {$\alpha_3$} [b] at 152 154
\pinlabel {$\alpha_4$} [b] at 172 154
\pinlabel {$\beta_1$} [t] at 126 10
\pinlabel {$\beta_2$} [t] at 162 10
\pinlabel {$\gamma_1$} [l] at 230 82
\pinlabel {$\gamma_2$} [b] at 252 38
\endlabellist
\includegraphics[scale=.75]{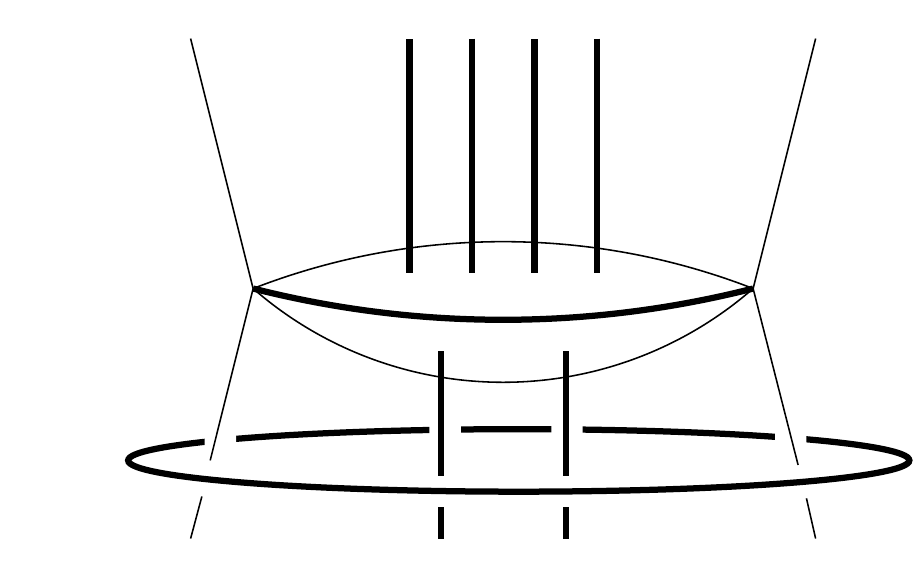}
\end{center}
\caption{Cone manifold as in  \fullref{thm:asymptotic}. The interior
of $K$ is not described in this picture and the singular locus is
represented thicker.  The angle defects satisfy
$\sum(2\pi-\alpha_i)=\sum(2\pi-\beta_i)=2\pi$ and
$\sum(2\pi-\gamma_i)\leq\pi$}
\end{figure}
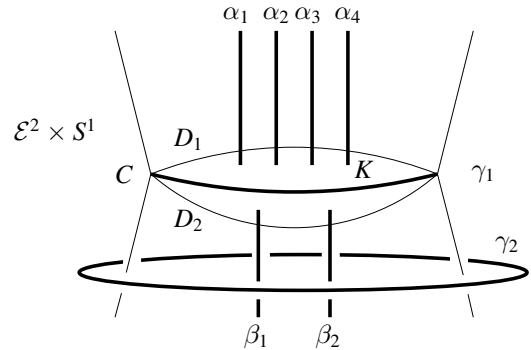

 It can happen that both discs are the same: $K=D_1=D_2$, as in \fullref{ex1}.

\begin{cor}
\label{cor:2pi/3} Let $E$ be a cone manifold as in
\fullref{thm:asymptotic}. If the cone angles are $<\frac{3\pi}2$,
then $E$ is as in \fullref{ex1}.
\end{cor}

Notice that in the case of \fullref{ex1} the topology of the
singular locus is the simplest one. In particular it is the case
when cone angles are $<\frac{3\pi}2$.

The topology of \fullref{ex2}   is still easy to understand in
terms of rational 2--tangles. We shall illustrate in
\fullref{sec:edgepatterns} that when cone angles are
$\frac{3\pi}2$ the topology may be more involved.
We shall describe cone manifolds with all cone angles precisely equal to 
$\frac{3\pi}{2}$.

\subsection{Organization of the paper} In \fullref{sec:euclideancm} 
we recall the basic properties for Euclidean cone manifolds, including
the soul theorem of Cheeger and Gromoll. \fullref{sec:onetwo}
deals with the case of one or two dimensional soul. The zero dimensional case
and the proof of \fullref{thm:asymptotic} is the content of Sections~\ref{sec:zero}
and~\ref{sec:asymptotic}. Finally \fullref{sec:edgepatterns} deals with the case
where all cone angles equal $\frac{3\pi}{2}$.

\subsection*{Acknowledgements} This research was supported by the
Catalan government through grant 2003BEAI400228, 
{FEDER/MEC grant BFM2003-03458},
the CRM at Barcelona, UCSB and NSF grant DMS-0405963.

\section{Euclidean cone manifolds}
\label{sec:euclideancm}

A \emph{Euclidean cone $3$--manifold} $E$ is
a smooth $3$--manifold  equipped with a
metric so that it is a complete length space locally isometric to
\begin{itemize}
\item either the Euclidean space $\mathbb{R}^3$ (smooth points),

\item or a neighborhood of a singular edge (singular points).
\end{itemize}
The local model of the singular points is given,  in cylindrical
coordinates, by the following metric
\[
    d s^2 =
     dr^2+\left(\frac{\alpha}{2\pi} r\right)^2 d\theta^2+dh^2
     \]
where $r\in (0,+\infty)$ is the distance from the singular axis $\Sigma$,
$\theta\in [0,2\pi)$ is the rescaled angle parameter around $\Sigma$
and $h\in\mathbb{R}$ is the distance along $\Sigma$.
The angle $\alpha>0$ is called the  \emph{singular angle}.
When $\alpha=2\pi$ this is the standard smooth metric of  $\mathbb{R}^3$.

{According to our definition, the singular locus} $\Sigma$ is a submanifold of
codimension two and the cone angle is constant along each connected
component.
For cone manifolds in general one must allow singular vertices too.

We shall also consider two dimensional cone manifolds;
that is, by taking polar coordinates $(r,\theta)$
in the previous description of the singularity, so that the singular locus is discrete.
Isolated singular points are also called cone points.

\begin{rem}
Since we assume that the cone angles are less than $2\pi$, the cone
manifolds considered here are \emph{Alexandrov spaces of non-negative
curvature}, hence the corresponding comparison results apply (see
Burago--Burago--Ivanov \cite{BBI} and Burago--Gromov--Perel'man
\cite{BGP}): Toponogov comparison for triangles and hinges, the splitting
theorem, etc.
\end{rem}

For instance, using comparison results, in
Boileau--Leeb--Porti~\cite[Proposition~8.3]{BLP} it is proved:

\begin{prop}
The number of singular components of  a Euclidean cone 3--manifold
with cone angles $\leq c<2\pi$ is finite.
\end{prop}

\subsection{The soul and the Cheeger--Gromoll filtration}
We recall the construction of Cheeger--Gromoll filtration
(see Cheeger--Gromoll~\cite{CG} and Boileau--Porti~\cite{BP}).
Let $E$ be a non-compact Euclidean cone manifold without singular
vertices and cone angles $\leq c<2\pi$.
Given a point $p\in E$, we consider all rays $r\co[0,+\infty)\to E$ starting at $p$.
Since $E$ has non-negative curvature in the Alexandrov
sense, Busemann functions $b_{r}\co C\to\mathbb R$ are convex. 
{For every {$t\in\mathbb{R}$}, define}
$$
C_t=\{x\in E\mid b_r(x)\leq t\textrm{ for all rays } r\textrm{ starting at
}p\}.
$$
The sets $C_t$ are convex, give a filtration of $E,$ and {\color{black}if $C_{t_2}$ is not empty then 
for $t_2\geq t_1$}
$$
\partial C_{t_1}=\{x\in C_{t_2}\mid d(x, \partial C_{t_2})=t_2-t_1\}.
$$
{To construct the soul, we start with the smallest $t$ for which $C_t$ is not
empty.  Then $C_t$ is a convex manifold of dimension less than $3.$}
If this lower dimensional manifold has boundary, we
continue to decrease the set by taking distance subsets to the
boundary. We stop when we get a submanifold without boundary
(possibly a point), which is the soul $S$.

The sets of this filtration are totally convex, that is, no geodesic
segment with endpoints in this sets can exit them, even if the
geodesic is not totally minimizing. The fact that $S$ is totally
convex determines the topology {of the underlying space $\vert E\vert$,
which is a topological bundle over the soul with fiber $\mathbb{R}^k$, but 
not the topology of the singular locus.
The metric of $E$ is easy to describe when $\dim
S=2, 1$, as we discuss in the following section.}

\section{One or two dimensional soul}
\label{sec:onetwo}

Let $E$ be a non-compact Euclidean cone 3--manifold with cone angles $\leq
c<2\pi$, and denote by $S$ its soul.

\begin{prop}If $\dim(S)=2$, then $E$ is isometric to
\begin{itemize}
\item either a product of $\mathbb{R}$ with a Euclidean 2--sphere with cone
points,
\item or its orientable quotient, a bundle over the projective
plane with cone points.
\end{itemize}
\end{prop}

On the Euclidean 2--sphere   the cone angle defects add up to $4\pi$, and on
the projective plane, $2\pi$. In particular, the upper bound
$c<2\pi$ gives an explicit bound on the number of singular components.

\begin{proof} If the soul $S$ is orientable, then
$E$ has two ends and, by the splitting theorem, $E=S\times\mathbb{R}$. Since $S $ is a compact Euclidean surface, it must be a sphere
with cone points, whose angle defects add up to $4\pi$. If $S$ is
non-orientable, then $E$ is the orientable bundle over $S$, which
is a projective plane. By taking the double cover, we reduce to the previous case.
\end{proof}

\begin{prop}
\label{prop:dim1}
If $\dim(S)=1$, then $E$ is isometric to the metric suspension
of a rotation in a plane with cone points.
  \end{prop}

More precisely, $E$ is isometric to  
  $ [0,1]\times F^2/\sim$,
where $F^2$ is a plane with singular cone points, and $\sim$
identifies $\{0\}\times F^2$ with $\{1\}\times F^2$ by a rotation
(possibly trivial).

Again the cone angle defects on $F^2$ add up to {\color{black}$<$ $2\pi$}, hence the
{upper bound on the cone angles} $c<2\pi$ gives an explicit bound on the number of
singular components. If there are singular
components other than the suspension of a singular point in $F^2$
fixed by the rotation, then  the rotation has to be of finite order. In
{this case} {$\vert E\vert$ is homeomorphic} to a solid torus and the singular locus is made of
fibers of a Seifert fibration of  $\vert E\vert$ with at most one singular fibre
{(singular in the Seifert sense)}.

\begin{figure}
 \begin{center}
  \includegraphics[scale=.75]{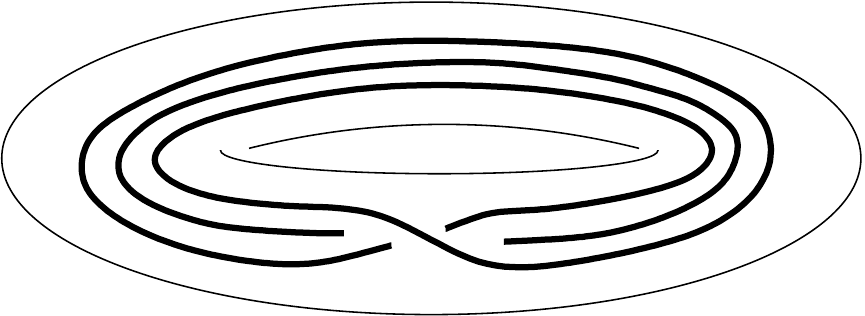}
  \end{center}
\caption{Cone manifold with soul $S^1$}
\end{figure}

\begin{proof}
Since $S$ is one dimensional, $S\cong S^1$. In particular
{$\pi_1   (\vert E \vert) \cong \mathbb{Z}$}. In the universal covering, $S$ lifts
to a line, hence by the splitting theorem $\tilde E=\mathbb{R}\times F^2$ for $F^2$ a two dimensional cone manifold. The
monodromy of the covering acts on $F^2$ by isometries, with a
fixed point corresponding to the soul.

Notice that $F^2$ is a non-compact Euclidean cone manifold with
non-empty singular locus. Hence the soul of $F^2$ is a point and $F^2$  is a
plane with several cone points.
\end{proof}

\section{Zero dimensional soul}
\label{sec:zero}

\begin{prop}
\label{prop:toproduct}
Let $\Sigma^{\mathit{noncpt}}$ denote the union of the non-compact components of
$\Sigma$. If $\dim(S)=0$, then the pair $(  \vert E\vert,
\Sigma^{\mathit{noncpt}})$ is homeomorphic to $\mathbb{R}^3$ with some
straight lines. {\color{black}The sum of the cone angle defects of these
non-compact components is}  $\leq 2\pi$.
\end{prop}

In particular, if $\Sigma$ has no compact components, the singular
locus is unknotted. Notice that again we have an explicit bound of
the number of non-compact components of $\Sigma$ coming from the
upper bound on the cone angles $\leq c<2\pi$.

\begin{proof}
Since the soul $S$ is a point, all sets of the Cheeger--Gromoll filtration
are balls. Those sets are totally convex, thus they intersect each
non-compact geodesic (singular or not) in precisely an interval
({\color{black}possibly empty}),
even if the geodesic is not { minimizing. 
This } applies to the
non-compact singular components of $\Sigma$, therefore, as we
increase the sets of the filtration, the singular components
intersect the sets  in segments that are increasing. This implies the
first assertion of the proposition.

For the assertion about the cone angle defects, notice that the
sets of the filtration have boundary with non-negative intrinsic
curvature, by convexity. The contribution of cone points to the
total curvature is larger than the cone angle defects, and we apply
Gauss--Bonnet.
\end{proof}

{At this point, one has a classification in the case when there are no compact components of the singular 
set. Now we start the proof of  \fullref{thm:asymptotic}, which is the main tool
in the classification when there are compact 
singular components and the cone angles are
either $<\frac{2\pi}3$ or all $=\frac{2\pi}3$. This proof occupies the remainder of this section
 and the following one. }

\subsection[Proof of \ref{thm:asymptotic}: finding the sector E^2 x R]{Proof of \fullref{thm:asymptotic}: finding the sector
$\mathcal{E}^2\times\mathbb{R}$}
 Let $C$ be
a closed singular geodesic in $E$. By comparison, every ray
starting at some point of $C$ must be perpendicular to $C$. In
addition, since the sets of the Cheeger--Gromoll filtration are
totally convex, either they contain $C$ or are disjoint from $C$.

\begin{lem}
\label{lem:flatstrip}
Every ray starting at $C$ is contained in a
flat half-infinite cylinder bounded by $C$.
\end{lem}

\begin{proof}
 Consider $r\co[0,+\infty)\to E$  a ray with
$r(0)\in C$.
Toponogov's  theorem applied to the triangle with edges $C$ and two
copies of $r([0,t])$,
when $t\to\infty$, gives that
$r$ and $C$ are perpendicular.

Next we show that $r$ can be parallel transported along $C$
 by another comparison argument.
Consider a parameterized subsegment $\sigma$ of $C$ starting at $r(0)$ of length $x$.
 For $t>0$, consider also a segment $\bar
\sigma$ starting at $r(t)$ parallel to $\sigma$  along $r$ of length $x$. Let
$s=d(\sigma(x),\bar \sigma(x))$. We know by comparison that $s\leq
t $ and we claim that $s=t$.

\begin{figure}[ht!]
\begin{center}
\labellist\small
\pinlabel {$C$} [l] at 24 107
\pinlabel {$\sigma(0)$} [r] at 22 73
\pinlabel {$\sigma(x)$} [r] at 22 47
\pinlabel {$\bar\sigma(0)$} [b] at 86 73
\pinlabel {$\bar\sigma(x)$} [t] at 86 47
\pinlabel {$s$} [t] at 56 47
\pinlabel {$t$} [b] at 56 73
\pinlabel {$d$} [b] at 142 73
\pinlabel {$\bar{d}$} [t] at 142 61
\pinlabel {$r(d{+}t)$} [b] at 192 73
\endlabellist
\includegraphics[scale=1]{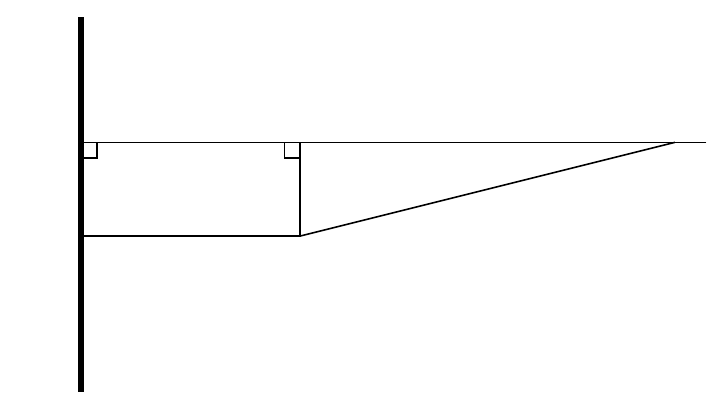}
\end{center}
\caption{Constructing the flat strip}
\label{fig:flatstrip}
\end{figure}
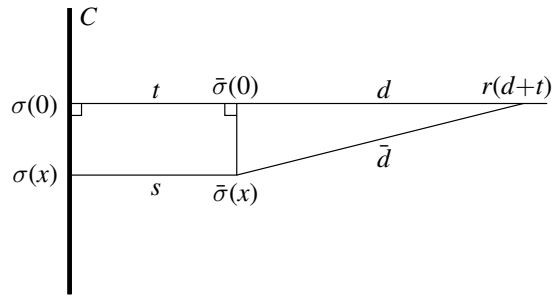

Seeking a contradiction, assume that $s<t$. Choose $d \gg 1$ and
set $\bar d= d(\bar\sigma(x), r(t{+}d))$. See \fullref{fig:flatstrip}. Since
$$
\lim_{d\to+\infty}\sqrt{d^2+x^2}-d=0,
$$
and $t-s>0 $ by hypothesis, applying comparison  $d$ can be chosen
large enough so that
$$
\bar d\leq  \sqrt{d^2+x^2} < d+t-s.
$$
Hence $d(r(d+t),C)\leq \bar d+s< d+t$. This would imply that $C$
has points inside and outside the sublevel zero set of the
Busemann function of $r$,  contradicting total
convexity.

This proves that the rays can be parallel transported along $C$.
We claim that this transport does not have monodromy, that is, once we
have made a whole turn around $C$, the ray is the same, so that it
gives a cylinder. We look
at the sublevel set $0$ for the Busemann function. For any vector $v$
tangent to this sublevel set, the angle between $v$ and $r$ is
$\geq \pi/2$. Hence there is no monodromy, otherwise this level
set would be two dimensional and $C$ smooth, but we are assuming
that the cone angle at $C$ is less than $2\pi$.
\end{proof}

For $p\in C$, all rays starting at $p$ are perpendicular to $C$,
hence the set of rays at $p$ lies in the set of directions
orthogonal to $C$, which is a circle of length equal to the cone
angle. In addition, those directions form an angle $\geq \pi/2$
with any direction tangent to the subset of the Cheeger--Gromoll
filtration. Thus it makes sense to talk about the two extremal or
outermost flat strips at $C$, which are the ones with larger
angle.

\begin{lem}
The two extremal flat strips at $C$ bound a metric  product
${\mathcal E^2}\times S^1$, where ${\mathcal E^2}$ is a 2--dimensional
Euclidean sector with cone points
and $C$ is the product of $S^1$ with the tip of the sector.
\end{lem}

\begin{proof}
We cut along both extremal flat strips, and consider the connected
component with angle $\leq c-\pi <\pi$. We glue the two half planes by
an isometry fixing $C$ pointwise. We call $Y$ this new Euclidean cone
manifold. The singular geodesic $C$ gives a singular geodesic in $Y$,
that we also denote by $C$. The new cone angle is the angle between the
strips, which is $<\pi$, since it is bounded above by the cone angle
of $C$ in $E$ minus $ \pi$ (the angle between any ray and the sets of
the Cheeger--Gromoll filtration is $\geq\pi/2$). Since $C$ is a closed
singular geodesic in $Y$ with cone angle $<\pi$, it must be contained in
the soul of $X$, because the convex hull of any point close to $C$ meets
$C$ (see Boileau--Porti~\cite[Lemma~4.2.5]{BP}). The soul must
be $C$ itself, since a two dimensional soul cannot contain
a closed singular geodesic. Thus $Y$ is a mapping torus as in
\fullref{prop:dim1}. The flat strip implies that actually $Y$
is a product, and the lemma is clear.
\end{proof}

\section{Asymptotic behavior}
 \label{sec:asymptotic}

In the previous section we found the factor ${\mathcal E^2}\times S^1$.
Continuing the proof of \fullref{thm:asymptotic} and, 
in order to analyze the rest of the manifold, we remove the interior of the
product ${\mathcal E^2}\times S^1$. This space now has two ends, corresponding
to the two half-lines in the boundary of the sector ${\mathcal E^2}$. Let $X$
be one of the ends, that is, one of the unbounded components if we
further remove a convenient compact subset. We will not worry about which
compact subset we have removed to analyze $X$. 
{\color{black}In order to simplify the argument the proofs below use some
standard facts about cone manifolds with boundary. We leave it as an exercise
for the reader to check  that the proofs can also be done  by doubling $X$ along
its boundary.}

Notice that  at each point in $C\subset \partial X$ 
there is a \emph{single ray}
going to $\infty$, since we have chosen the outermost flat strips.
Thus the Tits boundary of $X$ is a single point. The fact that the
Tits boundary does not depend on the base point in $X$ implies that
for any two rays in $X$ with $r_1(0)=r_2(0)$,
$$
\lim_{t\to+\infty}\frac{d(r_1(t),r_2(t))}{t}=0.
$$
Also, by looking at equivalent definitions (see Guijarro--Kapovitch~\cite{GK}), if $S_R\subset X$ is
the metric
sphere of radius $R$ centered at some fixed point,
\begin{equation}
\label{eqn:diamSR}
\lim_{R\to+\infty} \frac{\operatorname{diam}(S_R)}R=0.
\end{equation}

\begin{lem}
\label{lem:anglesing}
 Fix a point $p\in X$. For any $q\in X$, as $d(p,q)\to\infty$,
the angle between a minimizing segment to $p$ and a ray starting at $q$ goes to $\pi$.
\end{lem}

\begin{proof}
Let $r$ be any ray starting at $q$, and $\sigma$ any
minimizing segment between $p$ and $q$.

For large $t$, choose a segment $\sigma'$ between $p$ and
$r(t)$.
The segments $\sigma$, $\sigma'$ and a piece of
$r$ form a triangle with vertices $p$, $q$ and $r(t)$.
We want to show that its angle at $q$ is close to $\pi$. Let $q'$
be the point of $\sigma'$ such that $d(p,q)=d(p,q')$. By the
triangle inequality:
$$
\vert d(q',r(t))-t\vert \leq d(q,q').
$$
Thus
$$
\vert d(p,r(t))-d(p,q)-t\vert \leq d(q,q').
$$
By Equation~\eqref{eqn:diamSR}, the ratio $d(q,q')/d(p,q)$ is arbitrarily small,
(independently of the choice of $t$ and
$\sigma'$). By choosing $t$   arbitrarily large,  comparison implies that the
angle at $q$ between $\sigma$ and $r$ is arbitrarily close to
$\pi$.
\end{proof}

\begin{figure}[ht!]
\begin{center}
\labellist\small
\pinlabel {$p$} [r] at 5 8
\pinlabel {$q$} [b] at 76 35
\pinlabel {$q'$} [t] at 82 20
\pinlabel {$\sigma$} [br] at 44 22
\pinlabel {$\sigma'$} [t] at 116 25
\pinlabel {$r$} [l] at 203 35
\pinlabel {$r(t)$} [b] at 174 33
\endlabellist
\includegraphics[scale=1]{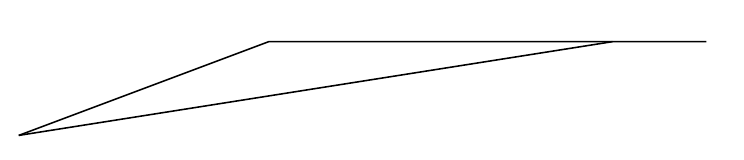}
\end{center}
\caption{The triangle $p$, $q$, $r(t)$ in the proof of \fullref{lem:anglesing}}
\end{figure}

\begin{cor}
\begin{enumerate}
\item The angle at $q$ between  any two minimizing segments to $p$ goes to 0.
\item The angle at $q$ between  any two rays starting at $q$ goes to 0.
\item If $q\in\Sigma\cap X$, then the angle between $\Sigma$ and any minimizing segment to $p$
goes to $0$. The angle between $\Sigma$ and any ray starting at $q$ goes to $0$.
\end{enumerate}
 \end{cor}

\begin {proof}
Assertions (1) and (2) are straightforward. To prove (3) we use
the upper bound on the cone angle $c<2\pi$.  The directions of
$\sigma$ and $r$ are arbitrarily close to the singular
directions, because this is the only way two directions of angle
close to $\pi$ can fit in the space of directions.
\end{proof}

\begin{cor} For any sequence $q_n\to\infty$, the
limit of pointed cone manifolds $(X,q_n)$ contains a line.
\end{cor}

\begin{proof}
Consider two points at distance $d_n$ from $q_n$, one on a ray
starting at $q_n$ and the other one in a minimizing segment to the
base point $p$. They form a triangle whose angle at $q_n$ goes to
$\pi$. Choose $d_n\to \infty$ depending on this angle so that the
distance between $q$ to the opposite edge of the triangle is
bounded. Thus it gives a line at the limit. 

Alternatively, \fullref{lem:anglesing} implies that the slope of the Busemann function 
restricted to the segment $\overline{pq}$ goes to one. Hence
the union of $\overline{pq}$ with a ray converges to a line.
\end{proof}

It follows from this corollary and the splitting theorem that
the {pointed Gromov--Hausdorff limit}
$\lim (X,q_n)=X_\infty$ is a product. We need however to
understand the behavior of the singular locus. We choose $q_n$ so
that the distance to the singularity is 1, and $q_n$ is contained in a 
{\color{black}{\em parallel copy of $C$} ie.\
a  smooth closed geodesic parallel to $C$.}

\begin{prop}
\label{prop:limits} The limit $X_\infty$ is a cone manifold
$X^2\times \mathbb{R}$, { where   $X^2$ is a disc 
with singular points and with boundary}
a parallel copy of $C$.
In addition,
the singular components of the
approximates become parallel to give the cone points of $X^2$, so
that when singular components merge at the limit then  the cone
angle defects add. 
\end{prop}

If the distance between singular components is bounded below away
from zero, then   $X_\infty$ is a cone manifold
and the argument for \fullref{prop:limits} is easy 
(see the proof below). 
Thus, as a preliminary step to prove this proposition, 
we need to understand how the singularities behave at the limit.

Denote by $\Sigma^1,\ldots,\Sigma^k$ the singular components of
$X$.  Take $x^i_n$ to be the intersection of $\Sigma^i$ with
the same level set as $q_n$ for the Busemann function. Assume that
$d(x^1_n,x^2_n)\to 0$ faster than the other $d(x^i_n,x^j_n)$,
that is, $d(x^1_n,x^2_n)\leq d(x^i_n,x^j_n)$.
 We take the rescaled limit
$$
\lim\biggl(\frac1{d(x^1_n,x^2_n)}X,x^1_n\biggr)= (X^1_{\infty},x^1_\infty).
$$
The singular component $\Sigma^1$ becomes a line at the limit
$X^1_{\infty}$, because by \fullref{lem:anglesing} the slope of the
Busemann function on $\Sigma^1$ converges to one.

\begin{lem}\label{lem:injx1}
The injectivity radius {\color{black}in $\frac1{d(x^1_n,x^2_n)}X$} at $x^1_n$ is bounded below away from zero.
\end{lem}

\begin{proof}
Consider a small ball centered at $x^1_n$. It is a metric ball 
with a singular diameter (which is a piece of $\Sigma^1$).  
Increase its radius 
until the ball intersect itself or meets a singularity along its boundary;
the radius of this ball is the injectivity radius. We control it
 by finding
lower bounds for
the length of geodesic paths $\gamma$ joining
$x^1_n$ to $\Sigma$ and 
{for} the length of geodesic loops $l$ with base point
$x^1_n$.

Let $\gamma$ be a geodesic path joining $x^1_n$ to another
singular point $y_n$, so that $\gamma$ itself is not contained in
the singular locus. By taking the shortest one, we may assume that
$\gamma$ is perpendicular to $\Sigma$ at $y_n$. By
\fullref{lem:anglesing} and triangle comparison, $\gamma$ is
almost perpendicular to every ray. Thus the Busemann function
restricted to $\gamma$ is almost constant. Since the Busemann
function restricted to the singular components has slope close to
one, if $y_n\in\Sigma^1$ then the length of the singular segment
$\overline{x_n^1\, y_n}\subset \Sigma^1$ is much shorter than 
the length $\vert
\gamma\vert$. When we compute the injectivity radius, we increase
the radius of a ball centered at $x^1_n$, and such a path does not
appear. If $y_n$ belongs to some other component $\Sigma^j$ of $\Sigma$,
then $y_n$ has to be close to the corresponding $x^j_n$. 
Namely, using that the Busemann function $b_r$ 
restricted to $\gamma$ has slope less than $\frac{1}{2}$, and restricted to 
$\Sigma^j$ more than $\frac 12$:
$$
d(y_n,x_n^j)\leq 2\,\vert b_r(y_n)-b_r(x^j_n)\vert = 2\, \vert b_r(y_n)-b_r(x^1_n)\vert
\leq \vert\gamma\vert.
$$
Thus 
$$
\vert\gamma\vert\geq d(x^1_n, y_n)\geq d(x^1_n,x_n^j)- d(x_n^j, y_n)\geq 1-\vert \gamma\vert,
$$
and
$\vert\gamma\vert\geq \frac12$.

Given a short geodesic loop $l$ with base point $x^1_n$, by
comparison and \fullref{lem:anglesing}, $l$ is almost
perpendicular to $\Sigma^1$. Let $\alpha$ be the angle of $l$ at
the base point. Since $c<2  \pi$ is the upper bound of the cone
angle, almost perpendicularity gives another bound $\alpha\leq
c'/2<\pi$.
 By
pushing $l$ in the direction of the angle at $x^1_n$, if it does
not meet the singular set it shrinks to a point at distance
$$
\frac{\vert l\vert /2}{\cos(\alpha/2)}\leq \frac {\vert l\vert/2}{\cos(c'/2)},
$$
where $\vert l\vert$ denotes the length of $l$. So $\frac {\vert
l\vert/2}{\cos(c'/2)}\geq \frac 12$, which is the previous bound
for $\vert\gamma\vert $. This proves the claim.
\end{proof}

\begin{lem}\label{lem:x1conemanifold}
The limit $X^1_ {\infty}$ is a cone manifold.
\end{lem}

\begin{proof}[Proof of \fullref{lem:x1conemanifold}]
The  argument of \fullref{lem:injx1}
also gives control of the injectivity radius
at each $x^i_n$. Once the distance to the singular locus is
controlled, the argument with the loops gives an injectivity radius estimate  for points
in $\frac1{d(x^1_n,x^2_n)}X$
at distance at most $R$ from $x_n^1$, for some fixed $R>0$.
This estimate is uniform on $n$ and $R$.
\end{proof}

Some components $\Sigma^i$ remain at the limit $X^1_{\infty}$ (at least
$\Sigma^1$ and $\Sigma^2$), some other components go to infinity.
The components that remain at the limit are parallel.
We shall use this to prove that,  
when rescaled by another factor, they  converge to a
single component whose cone angle defect is the limit.

\begin{proof}[Proof of \fullref{prop:limits}]
We take limits inductively, according to the order of
convergence to zero of $d(x^i_n,x^j_n)$, and always taking
subsequences, so that we use information from previous steps. More
precisely, in the limit $X^1_{\infty}$ we obtain the singular components
$\Sigma^i$ such that the ratio
$$
\frac{d(x^1_n,x^i_n)}{d(x^1_n,x^2_n)}
$$
 is bounded; the other components go to infinity.
We take the pair of coefficients $i$, $j$ such that
$d(x^i_n,x^j_n)\to 0$ with the next order of convergence. If
$i,j\neq 1$ and $\frac{d(x^1_n,x^j_n)}{d(x^i_n,x^j_n)}\to \infty$,
we repeat the argument of \fullref{lem:x1conemanifold} for the base point $x^j_n$ and we do not
care about $x^1_n$.  Otherwise we can assume $i=1$ and take the
limit
$$
\lim\biggl(\frac1{d(x^1_n,x^j_n)}X,x^j_n\biggr)= (X^j_{\infty}, x^j_\infty).
$$
At the ball  $B(x^j_\infty,1)$ of radius $1$, the singularity of
the approximating balls
$$B(x^j_n,1)\subset \frac1{d(x^1_n,x^j_n)}X$$ 
is controlled and therefore $B(x^j_\infty,1)$ is a cone
manifold. By the product structure, $X^j_{\infty}$ is a
 cone manifold at the neighborhood of $\Sigma^j$.
The {sequence} $x^1_n$ converges to $x^1_{\infty}\in X^j_{\infty}$ at
distance one from  $x^j_\infty$. 
For any $y_{\infty}\in X^j_{\infty}$
not in $x^1_{\infty}\times \mathbb{R}$, 
if $y_{\infty}$ is smooth, then the approximates $y_n$ are at
bounded distance from the singularity. 
Otherwise, if $y_{\infty}$ is singular, then  the 
$y_n$ are at bounded distance from the other singular components, 
by the choice of the indices $i$ and $j$.
Thus   the arguments in 
Lemmas~\ref{lem:injx1} and~\ref{lem:x1conemanifold} may be used to say 
that $X^j_{\infty}$
is locally a cone manifold away from  $x^1_{\infty}\times \mathbb{R}$.

\textbf{Claim}\qua
$X^j_\infty-(x^1_{\infty}\times \mathbb{R})$ is a non-complete
product cone manifold.

To take the metric completion of $X^j_{\infty}-(x^1_{\infty}\times
\mathbb{R})$, we take an arbitrarily
small loop around $x^1_{\infty}$ . By looking at the approximates to $X^1_{\infty}$ in \fullref{lem:x1conemanifold},
hence by changing the base point and the scale factor, the holonomy of this
loop must be a rotation of angle $2\pi$ minus the sum of
angle defects. This follows from the product structure of  $X^1_{\infty}$, 
and the fact that the rotation angle of the holonomy
does not depend on the scale factor and the choice of the base point (the conjugacy class).
Hence the completion is a cone manifold, and the
cone angle defect of the singularity is the sum of cone angle
defects of singular components merging with $\Sigma^1$. This proves the claim.

 We iterate
this process, which must stop by the finiteness of the number of
singular components.

Recall that the base points $q_n$ are contained in parallel copies of $C$ 
and that the distance to the singular locus is one. Thus the argument of 
\fullref{lem:injx1} applies to say that the injectivity radius at
$q_n$ is bounded below.
 This implies that the 2--dimensional factor $X^2$ in the limit
is a cone manifold containing at least one cone point and one
boundary component which is a geodesic circle. Therefore $X^2$
must be compact. Notice that the choice of base points $q_n$ does
not allow all singularities to merge to a single one, because this
would make the length of $C$    go to zero.
\end{proof}

It follows from \fullref{prop:limits} that the singular components are
asymptotically parallel. We claim that they are actually parallel.

\begin{prop}
Away from a compact set $X$ is a {\color{black}metric} product.
\end{prop}

\begin{proof}
By the Cheger--Gromoll filtration,  away from a compact set the 
pair formed by {$\vert X\vert $} and its singular locus is topologically a product
(\fullref{prop:toproduct}). By \fullref{prop:limits}, the
singular axis at $(X,q_n)$ are almost parallel.

We shall use the direct product  structure of the isometry group
$$\operatorname{Isom}^+(\mathbb{R}^3)=\mathbb{R}^3\rtimes \SO(3),
$$
 and
the fact that  the   holonomy lifts to $\mathbb{R}^3\rtimes \Spin(3)$ (see
Culler~\cite{C}). We
identify $\Spin(3)\cong S^3$ equipped with the standard round metric,
so that a rotation  in $\SO(3)$  of angle $\alpha\in [-\pi,\pi]$
lifts to two points in $S^3$, which are at respective distance
from the identity $\vert\alpha\vert/2$ and
$\pi-\vert\alpha\vert/2$.

The fundamental group of the smooth part $\pi_1(X-\Sigma)$
 is a free group generated by the meridians, so  we can
choose the lifts of their holonomy. In fact we will only look at
its projection to $\Spin(3)$, that we denote by
$A_1,\ldots,A_k\in \Spin(3)$.

For each meridian (with label $i$), if $\alpha_i$ is its cone angle defect, we choose
$A_i$ so that the distance to the identity in $S^3$ is
$\vert\alpha_i\vert/2$. The product of all meridians gives the
holonomy of $C$, which is a pure translation, so the product
$A_1\cdots A_k$
gives either the identity or its antipodal point in $S^3$,
at distance $\pi$.

We consider the piecewise geodesic path $\gamma$ in $\Spin(3)$ with ordered
vertices $ \operatorname{Id}$, $A_1$, $A_1\, A_2$,
\ldots , $A_1\, A_2\cdots A_k$. Notice that the angles
along $\gamma$ may depend on the conjugacy classes of the meridians,
but not the length of the pieces, because they are precisely half the
cone angle defects. By \fullref{prop:limits}, and possibly after
reordering the indexes, the conjugacy classes may be chosen so that
the angles along the path $\gamma$ are arbitrarily close to $\pi$.
Since the length of the pieces of $\gamma$ are fixed, the endpoint
of $\gamma$ cannot be the identity. Thus it is  antipodal to 
the identity. The bound on cone angle defects implies that the length of
$\gamma$ is at most $\pi$. Thus it is precisely $\pi$ and $\gamma$
is a geodesic, which implies that the singular axis are parallel.
\end{proof}

\begin{proof}[Conclusion of the Proof of \fullref{thm:asymptotic}]
Once we know that $X$ is a metric product away from a compact set,
we start to decrease the level set of the Busemann function until
it meets $C$. Hence \fullref{thm:asymptotic} is proved.
\end{proof}

\begin{proof}[Proof of \fullref{cor:2pi/3}]
We apply
\fullref{thm:asymptotic} to $E$. The dihedral angle between 
$D_1$ and $D_2$ in the boundary of the compact set $K$ is
less than the cone angle at $C$ minus $\pi$, because any direction
in $K$ has an angle $\geq \frac\pi2$ with any ray, and rays are
perpendicular to $C$. Thus, since the cone angle at $C$ is
$\alpha<\frac{3\pi}2$, the dihedral angle of the compact $K$ is
$<\alpha-\pi <\frac{\pi}2$. We claim that this angle is
zero. In order to show it, we analyze what happens to the singularity when
we shrink $K$ by taking distance subsets. Once the singularity in
one of the faces of $K$ reaches this boundary of the face, since
the singularity is perpendicular to the face, either it meets
immediately the other face (when the angle is zero), or it enters
the interior of $K$ (when the angle is larger than $\frac\pi2$). Since here
the dihedral angle is $<\frac\pi2$, it must be zero.
\end{proof}

In next section we will analyze what happens when the dihedral
angle of $K$ is precisely $\frac\pi2$, which implies working with cone
manifolds with cone angle $\frac{3\pi}2$. Larger cone angles would
probably give more complicated constructions.

\section{Constructions from an edge pattern on a parallelepiped}
\label{sec:edgepatterns}

We explain how to construct a non-compact Euclidean cone manifold
with cone angles $\frac{3\pi}2$ from an edge pattern on a rectangular parallelepiped.
 Such a cone manifold has soul a point and both closed and unbounded singular geodesics.

{\color{black}A rectangular parallelepiped is a subset of $\mathbb{R}^3$  isometric to the product of three finite closed intervals of positive length.
Consider the twelve edges of such a parallelepiped. } An \emph{edge pattern} consists of
joining the edges into intervals or circles, so that:
\begin{enumerate}
\item
The components of the pattern are one or
two circles and precisely four intervals.
 \item
 On each vertex, two of
the three edges are joined by the pattern.
\end{enumerate}

\begin{figure}[ht!]
\begin{center}
  \includegraphics[scale=.60]{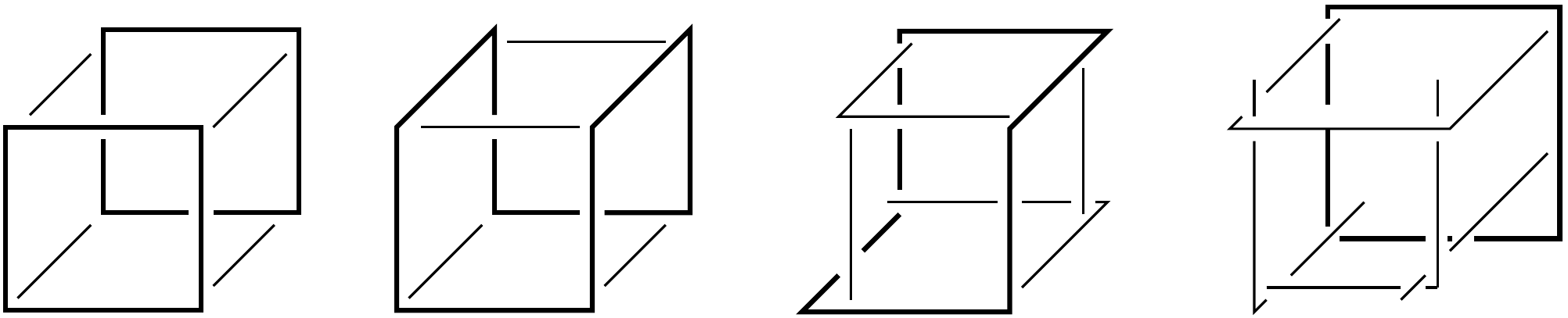} 
  \end{center}
   \caption{Examples of patterns, already realized metrically. The closed components are drawn thicker.}\label{fig:cubes1}
\end{figure}

\begin{figure}[ht!]
\begin{center}
  \includegraphics[scale=.75]{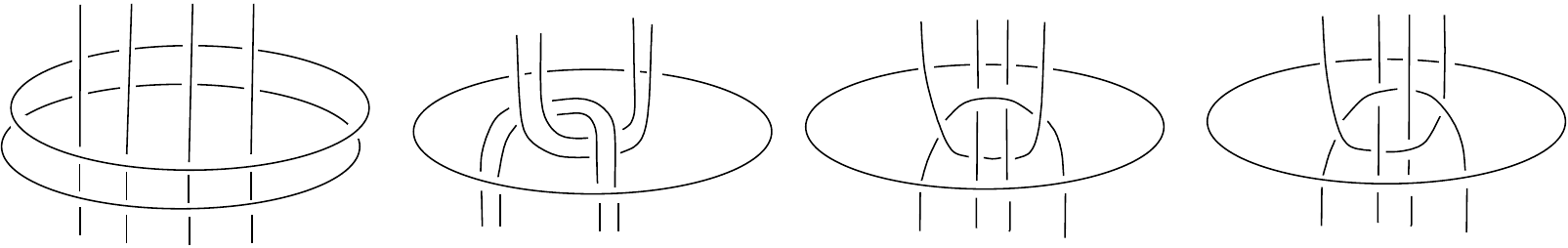}  
  \end{center}
   \caption{The singular locus of the respective cone manifolds of patterns in \fullref{fig:cubes1}, 
   respecting the order from left to rigth. The underlying space is $\mathbb{R}^3$ and all cone angles
   $\frac32\pi$.}\label{fig:linksfromcubes}
\end{figure}

We shall only consider patterns than are metrically realizable in
$\mathbb{R}^3$ satisfying the following properties. {\color{black}First we enlarge edges
as follows. Those
edges which have one free endpoint are enlarged to be unbounded geodesic rays and those 
with two free endpoints to be complete geodesics. Next we move by parallel
transport and slightly shorten  those edges which are not part of a circle in the pattern. 
This must be done so that at each corner the ray in the enlarged edge at that corner  lies
 inside the right angle defined by the other two segments joined in the pattern
at that corner. After this is done all the edges
must connect up to give the same combinatorial pattern.}
 Notice that this condition for a metric
realization already eliminates some patterns, see \fullref{fig:badpattern}.

\begin{figure}[ht]
\begin{center}
  \includegraphics[scale=.75]{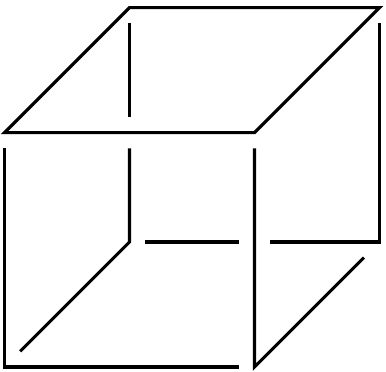}  
  \end{center}
   \caption{Example of pattern that cannot be realized metrically in $\mathbb{R}^3$. Notice that the
   edges must be parellel to coordinate axis.}\label{fig:badpattern}
\end{figure}

The cone manifold is constructed as follows. We   define a
subset in $\mathbb{R}^3  $ with piecewise geodesic boundary so that
the cone manifold is obtained by gluing its faces.  {\color{black}To each (possibly extended) 
edge $I$ in
the metric pattern, remove a {\em sector} of angle $\frac{\pi}2$ from
$\mathbb{R}^3$. This sector is a metric subset $I\times Q$
 bounded by planes parallel to the
sides of the parallelepiped. The set $Q$ is a quadrant in the plane
orthogonal to $I.$  
There are four such sectors and we chose the  one which is
in the opposite quadrant to the parallelepiped.  
The condition that the geodesic rays lie inside the right angle defined by segments
joined in the pattern implies that all the corners of the pattern
are removed. } The identifications consist of folding each such sector,
that is, in gluing the faces by a rotation. Thus the edges of the
pattern give the singular locus.

\begin{dfn}
Such a cone manifold is said to be \emph{constructed from an edge
pattern on a parallelepiped}.
\end{dfn}

Notice that two examples on the left of \fullref{fig:linksfromcubes}
are already described in Examples~\ref{ex1} and~\ref{ex2}, but the two
on the right are topologically different.

\begin{thm} Let $E$ be a Euclidean cone manifold with cone angles
$\frac{3\pi}2$, with soul a point and having a closed singular geodesic.
Then $E$ is constructed from an edge pattern in a parallelepiped.
\end{thm}

\begin{proof}
We apply \fullref{thm:asymptotic}.  Let $K$ be the compact
subset, whose boundary is a union of two singular discs along
their boundary, which is a singular geodesic $C$. Since we assume
that the cone angle is $\frac{3\pi}2$, the dihedral angle of $K$
is at most $\frac{\pi}2$. By the same argument as in
\fullref{cor:2pi/3}, if the dihedral angle is $<\frac\pi2$,
then $E$ is as in \fullref{ex1} and therefore it satisfies the
theorem. From now on we assume that the dihedral angle is
precisely $\frac{\pi}2$.

We shrink $K$  by considering $K_t$ supperlevel sets of the distance to
$\partial K$ (equivalently the sublevel sets of the Busemann function). 
{\color{black}Initially, for small $t,$
these subsets are bounded by the union of two faces, forming a
dihedral angle along the boundary. Each face is a disc with  four cone points.
The boundary of $K_t$ stays of this form as it shrinks 
until a cone point meets the boundary of the disc.}

Since we assume that the dihedral angle is $\frac\pi2$, when a
cone point meets the boundary {\color{black}of one of the faces $\partial K_t$} at
an edge, a whole segment of the singular
component has to lie in the other face of $\partial K_t$. If we
shrink further, we realize that a new edge on $\partial K_t$ has been
created for every pair of cone points going to the boundary of the
disc. Now the boundary is a union of flat cone manifolds, meeting
along edges with cone angle $\frac\pi2$, and edges meet at
corners, so that each corner is trivalent. If $n_{\mathit{cone}}$ and
$n_{\mathit{corner}}$ denote the number of cone points on the {\color{black}faces} and
corners respectively, then
\begin{align}
n_{\mathit{cone}}+n_{\mathit{corner}} &=8 \quad\text{globally on }\partial K_t \label{eqn:global}\\
n_{\mathit{cone}}+n_{\mathit{corner}} &=4 \quad\text{on each {\color{black}face} of }\partial K_t \label{eqn:side}
\end{align}
by the Gauss--Bonnet formula.

We continue the process of shrinking,
until some other cone point meets the boundary of the face.

It may happen that a cone point meets the boundary of the face at
a corner. By using \eqref{eqn:side}, the corresponding face must
be either a triangle with a cone point or a bigon with two cone
points. This face cannot have three cone points, because 
the distance between cone points stays constant during the shrinking, but
the face has to collapse. Thus the edges of the corner are different.

When a corner meets a cone point,  we change the process of shrinking, so that
the speed is not the same on each face. Hence  the shrinking
process becomes generic and we avoid cone points converging to a
corner.

So we assume that the cone points meet the boundary at the interior of an edge.
This creates new edges and corners, satisfying Equations~\eqref{eqn:global} and
\eqref{eqn:side}, until we end up
in one of the following situations:
\begin{enumerate}
\item[(a)] a smooth point,
\item[(b)] a singular point, or
\item[(c)] a one or two dimensional cone manifold with boundary.
\end{enumerate}
In case (a),  shortly before the collapsing time there are no cone
points at all, and formulas \eqref{eqn:global} and
\eqref{eqn:side} imply that $K_t$ must be combinatorially a cube,
hence $K_t$ is isometric to a parallelepiped. In case (b), the
same argument  gives a triangular prism with two cone points on
the upper and lower face. By changing the speed of the faces as
before, $K_t$ is non-singular, hence  a parallelepiped.

In case (c), some of the faces have collapsed. Since we assume
that cone points do not meet corners, the collapsing faces must be
rectangles. So shortly before the collapsing time $K_t$ must be a
product $X^2\times [0,\varepsilon]$ or $X^1\times
[0,\varepsilon]^2$, where $\dim X^i=i$. Notice that $X^1\times
[0,\varepsilon]^2$ is already a parallelepiped, and there is
nothing to prove. For $X^2\times [0,\varepsilon]$, the list of all
possible $X^2$ is quickly determined by \eqref{eqn:side}, and it
follows that changing the shrinking speed of the faces also gives
a parallelepiped. 
\end{proof}

\bibliographystyle{gtart}
\bibliography{link}

\end{document}